\documentclass[11pt,a4paper]{article}
\title{\bf Asymptotics of Kantorovich Distance for Empirical Measures of the Laguerre Model}
\author{Huaiqian Li \footnote{Email: {\color{blue}huaiqianlee@gmail.com}}
\quad Bingyao Wu\footnote{Email: {\color{blue}bingyaowu@163.com}}
  \vspace{2mm}
\\
{\footnotesize $^*$ Center for Applied Mathematics, Tianjin University, Tianjin 300072, China}\\
{\footnotesize $^\dag$ School of Mathematics and Statistics, Fujian Normal University, Fuzhou 350007, China}
}

\date{}
\usepackage{amssymb,amsmath,amsfonts,amsthm,color,mathrsfs}
\usepackage{bbm}
\usepackage[marginal]{footmisc}

\usepackage{xcolor} 

\def\R{\mathbb{R}}
\def\E{\mathbb{E}}

\def\L{\mathcal{L}}

\def\W{\mathrm{W}}
\def\idx{\mathbbm{1}}
\def\d{\textup{d}}

\def\CD{\textup{CD}}

\def\<{\langle}
\def\>{\rangle}
\def\Proof.{\noindent{\bf Proof. }}

\def\newdot{{\kern.8pt\cdot\kern.8pt}}

\newtheorem{theorem}{Theorem}[section]

\theoremstyle{definition}\newtheorem{remark}[theorem]{Remark}

\begin{document}
\allowdisplaybreaks
\maketitle
\makeatletter 
\renewcommand\theequation{\thesection.\arabic{equation}}
\@addtoreset{equation}{section}
\makeatother 

\begin{abstract}
We estimate the rate of convergence for the Kantorovich (or Wasserstein) distance between empirical measures of i.i.d. random variables associated with the Laguerre model of order $\alpha$ on $(0,\infty)^N$  and their common law, which is not compactly supported and has no rotational symmetry. Compared with the Gaussian case, our result is sharp provided the parameter $\alpha$ and the dimension $N$ are chosen in a specified regime.
\end{abstract}



\section{Introduction}\hskip\parindent
Let $\mathbb{M}$ be a Polish space and let $({\rm Z}_n)_{n\in\mathbb{N}}$ be a sequence of i.i.d. $\mathbb{M}$-valued random variables with common distribution $\mathfrak{m}$. Define the empirical measures associated with $({\rm Z}_n)_{n\in\mathbb{N}}$ as
$$\mathfrak{m}_n=\frac{1}{n}\sum_{k=1}^n \delta_{{\rm Z}_k},\quad n\in\mathbb{N},$$
where $\delta_\cdot$ stands for the Dirac measure. It is well known that, by the Glivenko--Cantelli theorem (see e.g. Theorem 7.1 on page 53 of \cite{Par1967}), almost surely, the sequence of random probability measures $(\mathfrak{m}_n)_{n\in\mathbb{N}}$ converges weakly to $\mathfrak{m}$ as $n\rightarrow\infty$. So, what about the rate of convergence and in what sense?

Let $d$ be the metric on $\mathbb{M}$ inducing the given topology of $\mathbb{M}$ and let $\mathcal{P}(\mathbb{M})$ be the class of all Borel probability measures on $\mathbb{M}$. For every $p\in[1,\infty)$, define the (pseudo) Kantorovich (or Wasserstein) distance $\W_p:\mathcal{P}(\mathbb{M})\times\mathcal{P}(\mathbb{M})\rightarrow[0,\infty]$ as the infimum
$$\W_p(\mu,\nu):= \inf_{\pi\in \Pi(\mu,\nu)} \Big(\int_{\mathbb{M}\times \mathbb{M}} d(x,y)^p\, \pi(\d x,\d y) \Big)^{1/p},\quad \mu,\nu\in \mathcal{P}(\mathbb{M}),$$
where $\Pi(\mu,\nu)$ is the class of all Borel probability measures on the product space $\mathbb{M}\times \mathbb{M}$ with
marginals $\mu$ and $\nu$,  respectively. For $p\in[1,\infty)$, let
$$\mathcal{P}_p(\mathbb{M})=\Big\{\mu\in\mathcal{P}(\mathbb{M}):\ \int_\mathbb{M}d(x,o)^p\, \mu(\d x)<\infty\Big\},$$
where $o$ is an arbitrary point in $\mathbb{M}$. It is well known that, for each $p\in[1,\infty)$, convergence in the Kantorovich distance $\W_p$ implies the weak convergence of sequences from $\mathcal{P}_p(\mathbb{M})$ in general, and if $\mathbb{M}$ is compact, then $\W_p$ metrizes the weak topology; see e.g. \cite[Theorem 6.9 and Corollary 6.13]{Villani2008}. See also \cite{Villani2008} for more details on the Kantorovich distance and its connection to the optimal transport theory.

So, it is natural to consider the (often challenging) problem of quantifying the rate of convergence of $\W_p(\mathfrak{m}_n,\mathfrak{m})$ as $n\to\infty$.

It is reasonable to see that a solution to this problem involves the distribution $\mathfrak{m}$ and also the geometry and topology of the space $(\mathbb{M},d)$. Being a classic subject with numerous applications, this kind of problem has attracted intensive investigation; see e.g. \cite{eAST,FG2015,BLG2014,DSS2013,AKT1984} for the i.i.d. case, \cite{Riek22,FG2015} for the Markov chain case, also
\cite{LiWu,LiWu2,LiWu3,eW1,WangWu21,fywany22,WZ19} for the case of continuous-time (subordinated) diffusion processes on compact or noncompact Riemannian manifolds with or without boundary and \cite{eW5} for the non-symmetric case, and even \cite{HMT} for the case of (non-Markovian) fractional Brownian motions on the flat torus and \cite{LiWu4} for the subordinated case.

The aim of the present work is to quantify the Kantorovich distance between empirical measures of i.i.d. random variables associated with the Laguerre model and their common distribution. The motivation comes from recent studies on the i.i.d. case associated with the Ornstein--Uhlenbeck (abbrev. OU) or Hermite model \cite{Ledoux3,Ledoux1} and the Jacobi model \cite{Zhu21}, respectively. The OU case corresponding to the standard Gaussian distribution which is rotationally symmetric and the Jacobi case corresponding to the Beta distribution which has compact support. However, the Laguerre case corresponding to the exponential distribution which has neither the compact support nor the rotational symmetry.

We should mention that, the Laguerre model is naturally associated with the Laguerre polynomial, which is one of the four classes of classical orthogonal polynomials and plays an important role in physics and mathematics; see \cite{Lebe}. Moreover, being a typical example of the Sturm--Liouville operator, the Laguerre model has been studied intensively in analysis and probability; see e.g. \cite{PS2021,NSS2017,BGL2014,Nowak04} and references therein.

In Section 2, we recall the Laguerre model and introduce the main results, and in Section 3, we present proofs of our main results.

\section{Preliminaries and main results}\hskip\parindent
We begin with the introduction of the Laguerre model; see \cite{NSS2017,BGL2014,Nowak04} for more details.

Let $N\in\mathbb{N}$  and $\R_>=(0,\infty)$. Let $\alpha=(\alpha_1,\cdots,\alpha_N)$ be a multi-index from $(-1,\infty)^N$. For each $x\in \R_>^N$, we always write $x=(x_1,\cdots,x_N)$.  Consider the probability measure given by
$$\mu^\alpha(\d x)=\prod_{i=1}^N \frac{x_i^{\alpha_i}e^{-x_i}}{\Gamma(1+\alpha_i)}\,\d x,$$
where $\Gamma(\cdot)$ stands for the Gamma function and $\d x$ denotes the Lebesgue measure on $\R^N$. The Laguerre (differential) operator of type $\alpha$, denoted by
$$\L^\alpha=\sum_{i=1}^N\big[x_i\partial_{x_i}^2 + (\alpha_i+1-x_i)\partial_{x_i}\big],$$
is non-positive and symmetric in $L^2(\R^N_>,\mu^\alpha)$, and $\L^\alpha$ has a self-adjoint extension in $L^2(\R^N_>,\mu^\alpha)$ which is denoted by the same notation.

Let $\mathbb{N}_0=\{0\}\cup\mathbb{N}$.  Consider the $N$-dimensional Laguerre polynomials of type $\alpha$, i.e.,
$$L^\alpha_n(x)=\prod_{i=1}^N L^{\alpha_i}_{n_i}(x_i),$$
for every $n=(n_1,\cdots,n_N)\in\mathbb{N}_0^N$ and every $x=(x_1,\cdots,x_N)\in \R_>^N$,
where for each $i=1,\cdots,N$,  $L^{\alpha_i}_{n_i}(x_i)$ is a one-dimensional Laguerre polynomial with degree $n_i$ of type $\alpha_i$ given by
$$L_{n_i}^\alpha(x_i)=\frac{1}{n_i!}x_i^{-\alpha_i}e^{x_i}\frac{\d^{n_i}}{\d x_i^{n_i}}\big(x_i^{n_i+\alpha_i}e^{-x_i}\big).$$
For every $n\in\mathbb{N}_0^N$ with $n=(n_1,\cdots,n_N)$, set $n!:=\prod_{i=1}^N n_i!$ and  $|n|:=n_1+\cdots+n_N$.  Let
$$l^\alpha_n(x)=\sqrt{n!} \prod_{i=1}^N\frac{L_{n_i}^{\alpha_i}(x_i)}{\Gamma(1+\alpha_i+n_i)},\quad x=(x_1,\cdots,x_N)\in\R_>^N,$$
be the normalized Laguerre polynomials. It is well known that the family $\{l_n^\alpha:\ n\in\mathbb{N}_0^N\}$ is an orthonormal basis of $L^2(\R_>^N,\mu^\alpha)$, and each $l^\alpha_n$ is an eigenfunction of $-\L^\alpha$ with corresponding eigenvalue $|n|$, i.e.,
$$-\L^\alpha l_n^\alpha=|n|l_n^\alpha, \quad n\in \mathbb{N}_0^N.$$

The Laguerre semigroup generated by $\L^\alpha$ is denoted by $\{P_t^\alpha\}_{t\geq0}$, which is known to be a symmetric diffusion semigroup in the sense of Stein \cite[page 65]{Stein70} (see .e.g. \cite{Nowak04}) and can be represented by
$$P_t^\alpha f(x)=\int_{\R_>^N}p_t^\alpha(x,y)f(y)\,\mu^\alpha(\d y),\quad f\in L^p(\R_>^N,\mu^\alpha),\,t>0,$$
for all $p\in[1,\infty]$, where $\{p_t^\alpha\}_{t>0}$ is the Laguerre kernel given by
$$p^\alpha_t(x,y)=\sum_{k\in\mathbb{N}_0^N} l^\alpha_k(x)l^\alpha_k(y) e^{-|k|t},\quad x,y\in\R_>^N,\,t>0.$$
Indeed, $p^\alpha_t(x,y)$ is strictly positive and infinitely differentiable  w.r.t. $(t,x,y)$ in $(0,\infty)\times\R_>^N\times\R_>^N$, and more precisely (see e.g. \cite[page 78]{Lebe}),
\begin{eqnarray*}
p^\alpha_t(x,y)&=&\prod_{k=1}^N\Gamma(1+\alpha_k)\frac{\exp\big[-\frac{e^{-t}}{1-e^{-t}}(x_k+y_k)\big]}{(1-e^{-t})(e^{-t}x_ky_k)^{\alpha_k/2}}  I_{\alpha_k}\bigg(\frac{2\sqrt{e^{-t}x_ky_k}}{1-e^{-t}}\bigg)\\
&=:&\prod_{k=1}^N p^{\alpha_k}_t(x_k,y_k),\quad (t,x,y)\in\R_>^{1+2N},
\end{eqnarray*}
where for each $\lambda\in\R$, $I_\lambda$ is the modified Bessel function of the first kind (also called Bessel function of imaginary argument) with order $\lambda$, defined as
$$I_\lambda(x)=\sum_{j=0}^\infty \frac{1}{\Gamma(j+1)\Gamma(\lambda+j+1)}\Big(\frac{x}{2}\Big)^{\lambda+2j},\quad x\in\R_>;$$
see e.g. \cite[Section 5.7]{Lebe}. Moreover, by \cite[(4.5) and (4.6)]{Muck1969}, for each $k=1,\cdots,N$, there exist constants $C_k,c_k>0$ depending only on $\alpha_k$ such that
\begin{equation}\label{HUP-1}
c_k\phi_t(x_k,y_k)\leq p^{\alpha_k}_t(x_k,y_k)\leq C_k\phi_t(x_k,y_k),\quad (t,x_k,y_k)\in\R_>^3,
\end{equation}
where
\begin{equation}\label{HUP}
\phi_t(x,y):=
\begin{cases}
\frac{1}{(1-e^{-t})^{\alpha+1}}\exp\Big(-\frac{e^{-t}(x+y)}{1-e^{-t}}\Big),\quad&{0<xy<\frac{(1-e^{-t})^2}{4e^{-t}},}\\
\frac{(4e^{-t}xy)^{-\alpha/2-1/4}}{e(1-e^{-t})^{1/2}}\exp\Big(-\frac{e^{-t}(x+y)-2(e^{-t}xy)^{1/2}}{1-e^{-t}}\Big),
\quad&{xy\geq\frac{(1-e^{-t})^2}{4e^{-t}}.}
\end{cases}
\end{equation}
Indeed, \eqref{HUP} comes from standard estimates on the modified Bessel function of the first kind.

Now consider the case when $N=1$. Let $\alpha\in(-1,\infty)$. It is easy to see that, for every $f\in C^2(\R_>)$, the \emph{carr\'{e} du champ} associated with $\L^\alpha$ is given as follows, i.e.,
$${\rm G}^\alpha(f)(x):=\frac{1}{2}\big[\L^\alpha(f^2)-2f\L^\alpha f\big](x)=x\big[f'(x)\big]^2,\quad x\in\R_>.$$
Then $(\R_>,\mu^\alpha,{\rm G^\alpha})$ is a Markov Triple; see \cite{BGL2014} for more details. It is easy to see that the intrinsic metric induced by the Markov Triple $(\R_>,\mu^\alpha,{\rm G^\alpha})$ is given by
$$\varrho(x,y)=\left|\int_x^y \frac{1}{\sqrt{t}}\,\d t\right|=2|\sqrt{x}-\sqrt{y}|,\quad x,y\in\R_>.$$

Let $p\in [1,\infty)$ and let $\mathcal{P}(\R_>^N)$ be the class of all Borel probability measures on $\R_>^N$. On the product space $\R_>^N$, we endow the metric
$$\varrho_N(x,y)=\Big(\sum_{j=1}^N \varrho(x_j,y_j)^2\Big)^{1/2},\quad x=(x_1,\cdots,x_N)\in\R_>^N,\,y=(y_1,\cdots,y_N)\in\R_>^N.$$
We use $\varrho_N(0,\cdot)$ to denote the distance to the boundary $\{0\}$ of $\R_>^N$. Obviously, $\varrho_N=\varrho$ when $N=1$. For every $\mu,\nu\in\mathcal{P}(\R_>^N)$, recall the (pseudo) Kantorovich
 distance of order $p$ between $\mu$ and $\nu$, i.e.,
\begin{equation*}\label{Wp}
\W_p(\mu,\nu)=\Big(\inf_{\gamma\in\Pi(\mu,\nu) }\int_{\R_>^N\times\R_>^N}\varrho_N(x,y)^p\,\gamma(\d x,\d y)\Big)^{1/p}\in[0,\infty],
\end{equation*}
where $\Pi(\mu,\nu)$ is the set of all couplings of $\mu$ and $\nu$, i.e., the set of Borel probability
measures $\gamma$ on $\R_>^N\times\R_>^N$ with marginals $\gamma(A\times \R_>^N) = \mu(A)$ and $\gamma(\R_>^N \times B) = \nu(B)$ for all Borel subsets $A,B$ of $\R_>^N$. In what follows, we concentrate on the quadratic Kantorovich distance $\W_2$.

The main result is presented in the next theorem. In what follows, we employ the notation $a\preceq b$, which means that $a\leq c b$ for some positive constant $c>0$, and $c$ may be either numerical or depend only on the parameters $N,\alpha$ but never on $n$. And we write $a\asymp b$ if both $a\preceq b$ and $b\preceq a$ hold. For every $\alpha=(\alpha_1,\cdots,\alpha_N)\in\R^N$, set $\alpha_\star:=\alpha_1+\cdots+\alpha_N$.
\begin{theorem}\label{main-thm}
Let  $n\in\mathbb{N}$, $\alpha\in[-1/2,\infty)^N$ and $X_1,\cdots,X_n$ be independent random variables on $\R_>^N$ with common distribution $\mu^\alpha$. Set
$$\mu_n=\frac{1}{n}\sum_{j=1}^n \delta_{X_j}.$$
Then, for every large enough $n$,
\begin{eqnarray}\label{main-thm-1}
\E\big[\W_2(\mu_n,\mu^\alpha)^2\big]\preceq\left\{\begin{array}{ll}\medskip
\frac{\log n }{n^{1/(\alpha_\star+N)}},&\quad \alpha_\star\in(1-N,\infty),\\\medskip
\frac{(\log n)^2}{n},&\quad \alpha_\star=1-N,\\\medskip
\frac{(\log n)^{2(\alpha_\star+N)-1}}{n},&\quad \alpha_\star\in (1/2-N,1-N),\\\medskip
\frac{\log\log n}{n},&\quad \alpha_\star=1/2-N.
\end{array}
\right.
\end{eqnarray}
\end{theorem}

Some remarks are in order.
\begin{remark}
Compared with the main result in \cite{Ledoux1}, for the case when $N=1$ and $\alpha=-1/2$, the rate $(\log\log n)/n$ in \eqref{main-thm-1} coincides with the one-dimensional OU case, and for the case when $N=2$ and $\alpha=(\alpha_1,\alpha_2)=(-1/2,-1/2)$, the rate $(\log n)^2/n$ in \eqref{main-thm-1} coincides with the two-dimensional OU case. We should mention that the rate  $(\log n)^2/n$ is sharp for the OU case in the critical two-dimension situation which is obtained recently in \cite{CP2023}; refer also to the aforementioned paper for the precise renormalized limit. However, when $N\geq3$, the result in \eqref{main-thm-1} contains an extra factor $\log n$ which is worse than the OU case in \cite[Theorem 1.1]{Ledoux3}.
\end{remark}

\begin{remark} The ultra-contractivity, which plays an important role in \cite{Zhu21,eAST,WangWu21,fywany22,WZ19}, fails for $\{P_t^\alpha\}_{t>0}$. Indeed, for every $t>0$,
\begin{eqnarray*}\|P_t^\alpha\|_{L^1(\R_>^N,\mu^\alpha)\rightarrow L^\infty(\R_>^N,\mu^\alpha)}&:=&\sup_{\|f\|_{L^1(\R_>^N,\mu^\alpha)}\leq 1}\|P_t^{\alpha} f\|_{L^\infty(\R_>^N,\mu^\alpha)}\\
&=&\sup_{(x,y)\in\R_>^{2N}}p^\alpha_t(x,y)=\infty.
\end{eqnarray*}
Consider the $N=1$ case. Applying \eqref{HUP}, we have for any $x\geq\frac{1-e^{-t}}{2e^{-t/2}}$,
\begin{eqnarray*}
\phi_t(x,x)=\frac{(4e^{-t}x^2)^{-\alpha/2-1/4}}{e(1-e^{-t})^{1/2}}\exp\Big(\frac{2e^{-t/2}x(1-e^{-t/2})}{1-e^{-t}}\Big),
\end{eqnarray*}
and hence
$$\sup_{(x,y)\in\R_>^2}\phi_t(x,y)\geq\sup_{x\geq\frac{1-e^{-t}}{2e^{-t/2}}}\phi_t(x,x)=\infty,$$
which together with \eqref{HUP-1} clearly implies that $\sup_{(x,y)\in\R_>^2}p^\alpha_t(x,y)=\infty$.
\end{remark}

Due to technical issues, we are not able to obtain reasonable upper bound estimates on $\E[\W_p(\mu_n,\mu^\alpha)^p]$ for other $p$'s at the moment. The problem seems much more complicated than the OU case. So we leave it for future study.

\section{Proofs of the main result}\hskip\parindent
In this section, we present the proof of the main theorem. The proof is based on the standard truncation argument and the PDE approach introduced in \cite{Ledoux1,BGV} and \cite{eAST}, respectively.

\begin{proof}[Proof of Theorem \ref{main-thm}]
Let $R>0$ and let $B_R= B_R^1\times\cdots\times B_R^N$, where for each $k=1,\cdots,N$,
 $B_R^k:=\{y\in\R_>:\ \varrho(0,y)<R\}$,
i.e., $B_R^k=\{y\in\R_>:\ y<R^2/4\}$. Set
$$\mu_R^\alpha=\frac{\mathbbm{1}_{B_R}}{\mu^\alpha(B_R)} \mu^\alpha.$$
Let $Y_1,\cdots,Y_n$ be independent random variables with common distribution $\mu_R^\alpha$ and be also independent of $X_1,\cdots,X_n$.  For each $j=1,\cdots,n$, define
\begin{eqnarray*}
X_{j,R}=\begin{cases}
X_j,\quad&{X_j\in B_R,}\\
Y_j,\quad&{X_j\in \R_>^N\setminus B_R}.
\end{cases}
\end{eqnarray*}
Then $X_{1,R},\cdots,X_{n,R}$ are independent with common distribution $\mu^\alpha_R$. Set
$$\mu_{n,R}=\frac{1}{n}\sum_{j=1}^n\delta_{X_{j,R}}.$$
Let
$$f_{n,R,t}(\cdot)=\frac{1}{n}\sum_{j=1}^n p^\alpha_t(X_{j,R},\cdot),$$
and denote $\mu_{n,R,t}=f_{n,R,t}\mu^\alpha$.

In order to estimate $\E\big[\W_2(\mu_n,\mu^\alpha)^2\big]$, the idea is simply to write
\begin{eqnarray}\label{pf-main-1}
\E\big[\W_2(\mu_n,\mu^\alpha)^2\big]\preceq \E\big[\W_2(\mu_n,\mu_{n,R})^2\big]+\E\big[\W_2(\mu_{n,R},\mu_{n,R,t})^2\big]+\E\big[\W_2(\mu_{n,R,t},\mu^\alpha)^2\big],
\end{eqnarray}
 by applying the triangular inequality, and then estimate each term in the right hand side of \eqref{pf-main-1} which is presented respectively in \textsc{Part} (1)--(3) below.

\underline{\textsc{Part} (1)}. Estimate $\E\big[\W_2(\mu_n,\mu_{n,R})^2\big]$. By the convexity of $\W_2^2$ (see e.g. \cite[Theorem 4.8]{Villani2008}) or by noticing that $\frac{1}{n}\sum_{j=1}^n\delta_{X_j}(\d x)\delta_{X_{j,R}}(\d y)\in \Pi(\mu_n,\mu_{n,R})$, we have
\begin{eqnarray*}
\W_2(\mu_n,\mu_{n,R})^2&\leq&\frac{1}{n}\sum_{j=1}^n\varrho_N(X_j,X_{j,R})^2\\
&=&\frac{1}{n}\sum_{j=1}^n\varrho_N(X_j,Y_j)^2\idx_{\R_>^N\setminus B_R}(X_j)\\
&\leq&\frac{2}{n}\sum_{j=1}^n\big[\varrho_N(0,X_j)^2+\varrho_N(0,Y_j)^2\big]\idx_{\R_>^N\setminus B_R}(X_j).
\end{eqnarray*}
Then
\begin{eqnarray*}
\E\big[\W_2(\mu_n,\mu_{n,R})^2\big]&\leq&4\int_{\R_>^N\setminus B_R}\varrho_N(0,x)^2\,\mu^\alpha(\d x)\\
&=&16\int_{\R_>^N\setminus B_R}(x_1+\cdots+x_N)\prod_{j=1}^N\frac{x_j^{\alpha_j}e^{-x_j}}{\Gamma(1+\alpha_j)}\,\d x_1\cdots\d x_N\\
&=:&f(R),\quad R>0.
\end{eqnarray*}
Let
$$h(R)=\prod_{j=1}^N \int_{R^2/4}^\infty\frac{x_j^{\alpha_j+1}e^{-x_j}}{\Gamma(1+\alpha_j)}\,\d x_j,\quad R>0.$$
Noting that $B_R=B_R^1\times\cdots\times B_R^N$, we may easily check that, as $R\rightarrow\infty$, both functions $R\mapsto f(R)$ and $R\mapsto h(R)$ have the same order.  For each $j\in\{1,\cdots,N\}$, it is easy to see that the integral
 $$\int_{R^2/4}^\infty x_j^{\alpha_j+1}e^{-x_j} \,\d x_j$$
has the order of $R$ as $R^{2(\alpha_j+1)}e^{-R^2/4}$ when $R$ tends to $\infty$. Hence, as a function of $R$, $\E\big[\W_2(\mu_n,\mu_{n,R})^2\big]$ has at most the order $R^{2(\alpha_\star+N)}e^{-R^2/4}$ as $R\rightarrow\infty$, where $\alpha_\star:=\alpha_1+\cdots+\alpha_N$ as before.

Now choosing $R=2(C\log n)^{1/2}$ for some large enough constant $C>0$, we derive that there exists a constant $c>1$ such that
\begin{eqnarray}\label{part1}
\E\big[\W_2(\mu_n,\mu_{n,R})^2\big]\preceq \frac{1}{n^c},
\end{eqnarray}
for every large enough $n$.

\underline{\textsc{Part} (2)}.  Estimate $\E\big[\W_2(\mu_{n,R},\mu_{n,R,t})^2\big]$. 
Let $$\hat{\pi}(\d x,\d y)=\frac{1}{n}\sum_{j=1}^n \delta_{X_{j,R}}(\d x)p_t^\alpha(X_{j,R},y)\mu^\alpha(\d y),\quad t>0.$$ It is easy to check that $\hat{\pi}\in\Pi(\mu_{n,R},\mu_{n,R,t})$ for every $t>0$. Then
\begin{eqnarray*}
\W_2(\mu_{n,R},\mu_{n,R,t})^2&\leq&\frac{1}{n}\sum_{j=1}^n \W_2\big(\delta_{X_{j,R}}, p^\alpha_t(X_{j,R},\cdot)\mu^\alpha\big)^2\\
&\leq&\frac{1}{n}\sum_{j=1}^n \int_{\R_>^N}\varrho_N(X_{j,R},y)^2p^\alpha_t(X_{j,R},y)\,\mu^\alpha(\d y).
\end{eqnarray*}
Hence
\begin{eqnarray*}
\E\big[\W_2(\mu_{n,R},\mu_{n,R,t})^2\big]&\leq&\frac{1}{n}\sum_{j=1}^n \int_{\R_>^N} \int_{\R_>^N}\varrho_N(x,y)^2p^\alpha_t(x,y)\,\mu_R^\alpha(\d x)\mu^\alpha(\d y)\\
&\leq&\frac{1}{\mu^\alpha(B_R)} \int_{\R_>^N} \int_{\R_>^N}\varrho_N(x,y)^2p^\alpha_t(x,y)\,\mu^\alpha(\d x)\mu^\alpha(\d y)\\
&=&\frac{4}{\mu^\alpha(B_R)}\sum_{j=1}^N \int_{\R_>} \int_{\R_>}|\sqrt{x_j}-\sqrt{y_j}|^2 p^{\alpha_j}_t(x_j,y_j)\,\mu^{\alpha_j}(\d x_j)\mu^{\alpha_j}(\d y_j)\\
&=:&\frac{4}{\mu^\alpha(B_R)}\sum_{j=1}^N {\rm I}_j(t).
\end{eqnarray*}

For convenience, we omit the subscript $j$ in the expression of ${\rm I}_j$ and write
$${\rm I}(t)=\int_{\R_>} \int_{\R_>}|\sqrt{x}-\sqrt{y}|^2 p^{\alpha}_t(x,y)\,\mu^{\alpha}(\d x)\mu^{\alpha}(\d y).$$
Let $\phi(x)=\sqrt{x}$, $x\in\R_>$. Then, for every $t>0$,
\begin{eqnarray*}
{\rm I}(t)&=&\int_{\R_>} \int_{\R_>}|\phi(x)-\phi(y)|^2 p^{\alpha}_t(x,y)\,\mu^{\alpha}(\d x)\mu^{\alpha}(\d y)\\
&=&2\int_{\R_>} \int_{\R_>}\big[\phi^2(x)-\phi(x)\phi(y)\big] p^{\alpha}_t(x,y)\,\mu^{\alpha}(\d x)\mu^{\alpha}(\d y)\\
&=&2\Big[ \int_{\R_>} P_t^\alpha(\phi^2)\,\d\mu^{\alpha}-\int_{\R_>} (P_{t/2}^\alpha\phi)^2\,\d\mu^{\alpha}\Big]\\
&=&2\Big[ \int_{\R_>} \phi^2\,\d\mu^{\alpha}-\int_{\R_>} (P_{t/2}^\alpha\phi)^2\,\d\mu^{\alpha}\Big]\\
&=&-2\int_{0}^{t/2}\Big(\frac{\d}{\d s}\int_{\R_>}(P_s^\alpha\phi)^2\,\d\mu^{\alpha}\Big)\,\d s\\
&=&-4\int_{0}^{t/2}\int_{\R_>}(P_s^\alpha\phi)\L^\alpha(P_s^\alpha\phi)\,\d\mu^{\alpha}\d s\\
&=&2\int_{0}^{t/2}\int_{\R_>}{\rm G}^\alpha(P_s^\alpha\phi)\,\d\mu^{\alpha}\d s\\
&\leq&2\int_{0}^{t/2}\int_{\R_>}e^{-s}P_s^\alpha{\rm G}^\alpha(\phi)\,\d\mu^{\alpha}\d s\\
&=&2\int_{0}^{t/2}e^{-s}\,\d s\int_{\R_>} {\rm G}^\alpha(\phi)\,\d\mu^{\alpha}\\
&\le&1-e^{-t/2}\leq \frac{t}{2},
\end{eqnarray*}
 where in the second equality we used the symmetry $p^\alpha_t(x,y)=p^\alpha_t(y,x)$, in the third equality we used the semigroup property of $(P^\alpha_t)_{t>0}$,
 in the fourth equality we used the invariance of $P_t^\alpha$ w.r.t. $\mu^\alpha$, and in the first inequality we employed the Bakry--Ledoux gradient estimate, i.e.,
 $${\rm G}^\alpha(P_s^\alpha\phi)\leq e^{-s}P_s^\alpha{\rm G}^\alpha(\phi),\quad s>0,$$
 since $\L^\alpha$ satisfies the curvature-dimension condition $\CD(1/2,\infty)$ when $\alpha\in[-1/2,\infty)$ (see e.g. \cite{BGL2014}).

 Thus, for every $ t>0$ and for all large enough $R>0$ such that $\mu^\alpha(B_R)\geq1/2$,
 \begin{eqnarray}\label{part2}
\E\big[\W_2(\mu_{n,R},\mu_{n,R,t})^2\big]\leq\frac{2Nt}{\mu^\alpha(B_R)}\leq 4N t.
\end{eqnarray}

\underline{\textsc{Part} (3)}.  Estimate $\E\big[\W_2(\mu_{n,R,t},\mu^\alpha)^2\big]$. This part is more technical.
By \cite[Theorem 2]{Ledoux1} (see also \cite{eAST} and \cite{fywany22}), i.e.,
\begin{equation*}\label{W2UB1}
\W_2(h\mu^\alpha,\mu^\alpha)^2\leq 4\int_{\R_>^N}{\rm G}^\alpha\big((-\L^\alpha)^{-1}(h-1)\big)\,\d\mu^\alpha,\quad h\geq0,\,\mu^\alpha(h)=1,
\end{equation*}
we have
\begin{eqnarray}\label{pf-main-3-0}
\E\big[\W_2(\mu_{n,R,t},\mu^\alpha)^2\big]&\leq& 4\E\Big[\int_{\R_>^N}{\rm G}^\alpha\big((-\L^\alpha)^{-1}[f_{n,R,t}-1]\big)(y)\,\mu^\alpha(\d y)\Big]\cr
&=&4\E\Big[\int_{\R_>^N}(f_{n,R,t}-1)\int_0^\infty P_s^\alpha(f_{n,R,t}-1)\,\d s\,\d\mu^\alpha\Big]\cr
&=&4\E\int_0^\infty\int_{\R_>^N}\big[P_{s/2}^\alpha(f_{n,R,t}-1)\big]^2\,\d\mu^\alpha\,\d s\cr
&=&8\E\int_0^\infty\int_{\R_>^N}\big[P_s^\alpha(f_{n,R,t}-1)\big]^2\,\d\mu^\alpha\,\d s.
\end{eqnarray}

Let $y\in\R_>^N$. Set
$$g(y):=\frac{1}{n}\sum_{j=1}^n\big(p_t^\alpha(X_{j,R},y)-\E[p_t^\alpha(X_{j,R},y)]\big),\quad b(y):=\E[p_t^\alpha(X_{j,R},y)]-1.$$
Recalling that
$$f_{n,R,t}=\frac 1 n \sum_{j=1}^n p_t^\alpha(X_{j,R},\cdot),$$
we have
\begin{eqnarray}\label{pf-main-3-1}
f_{n,R,t}(y)-1=g(y)+b(y).
\end{eqnarray}

Since
\begin{eqnarray}\label{pf-main-3-2}
b(y)&=&\int_{\R_>^N}p_{t}^\alpha(x,y)\,\mu_R^\alpha(\d x)-1\cr
&=&\frac{1}{\mu^\alpha(B_R)}\int_{\R_>^N}[\idx_{B_R}-\mu^\alpha(B_R)]p_{t}^\alpha(x,y)\,\mu^\alpha(\d x)\cr
&=&\frac{1}{\mu^\alpha(B_R)}P_t^\alpha[\idx_{B_R}-\mu^\alpha(B_R)](y),
\end{eqnarray}
we have
$$P_s^\alpha b(y)=\frac{1}{\mu^\alpha(B_R)}P_{t+s}^\alpha[\idx_{B_R}-\mu^\alpha(B_R)](y).$$
By the invariance of $(P^\alpha_t)_{t>0}$ w.r.t. $\mu^\alpha$, it is immediate to see
\begin{eqnarray*}
\mu^\alpha(P_s^\alpha b)=\frac{1}{\mu^\alpha(B_R)}\int_{\R_>^N}P_{t+s}^\alpha[\idx_{B_R}-\mu^\alpha(B_R)]\,\d\mu^\alpha=0.
\end{eqnarray*}
Hence, by this and \eqref{pf-main-3-2}, we have
\begin{eqnarray}\label{pf-main-3-3}
\mu^\alpha[(P_{t+s}^\alpha \idx_{B_R})^2]&=&\int_{\R_>^N}\mu^\alpha(B_R)^2(1+P_s^\alpha b)^2\,\d\mu^\alpha\cr
&=&\mu^\alpha(B_R)^2\Big(1+\int_{\R_>^N}(P_s^\alpha b)^2\,\d\mu^\alpha\Big).
\end{eqnarray}
Since $X_1,\cdots,X_n$ are independent and have the common distribution $\mu^\alpha$, we derive
\begin{eqnarray*}
\E\big[(P^\alpha_sg(y))^2\big]&=&\E\Big\{\Big(\frac{1}{n}\sum_{j=1}^n\big[p_{t+s}^\alpha(X_{j,R},y)-\E p_{t+s}^\alpha(X_{j,R},y)\big]\Big)^2\Big\}\cr
&=&\frac{1}{n}\E\big([p_{t+s}^\alpha(X_{1,R},y)-\E p_{t+s}^\alpha(X_{1,R},y)]^2\big)\cr
&=&\frac{1}{n}\Big[\int_{\R_>^N}p_{t+s}^\alpha(x,y)^2\,\mu^\alpha_R(\d x)-\Big(\int_{\R_>^N}p_{t+s}^\alpha(x,y)\,\mu^\alpha_R(\d x)\Big)^2\Big]\cr
&=&\frac{1}{n}\Big[\int_{\R_>^N}p_{t+s}^\alpha(x,y)^2\,\mu^\alpha_R(\d x)-\frac{1}{\mu^\alpha(B_R)^2}\big(P_{t+s}^\alpha\idx_{B_R}(y)\big)^2\Big].
\end{eqnarray*}
Hence, by \eqref{pf-main-3-3}, the symmetry and the semigroup property, we deduce that
\begin{eqnarray}\label{pf-main-3-4}
&&\int_{\R_>^N}\E\big[(P^\alpha_sg(y))^2\big]\,\mu^\alpha(\d y)\cr
&=&\frac{1}{n}\int_{\R_>^N}\int_{\R_>^N}p_{t+s}^\alpha(x,y)^2\,\mu^\alpha_R(\d x)\,\mu^\alpha(\d y)
-\frac{1}{n}\int_{\R_>^N}\Big(1+\int_{\R_>^N}(P_s^\alpha b)^2\,\d\mu^\alpha\Big)\,\d\mu^\alpha\cr
&\le&\frac{1}{n}\int_{\R_>^N}[ p_{2(t+s)}^\alpha(x,x)-1]\,\mu^\alpha_R(\d x)- \frac{1}{n}\int_{\R_>^N}(P_s^\alpha b)^2\,\d\mu^\alpha.
\end{eqnarray}
Combining \eqref{pf-main-3-1} and \eqref{pf-main-3-4}, we obtain
\begin{eqnarray}\label{pf-main-3-5}
&&\int_{\R_>^N}\E\big[P_s^\alpha (f_{n,R,t}-1)(y)\big]^2\,\mu^\alpha(\d y)\cr
&=&\int_{\R_>^N}\E\big[(P_s^\alpha g(y))^2+(P_s^\alpha b(y))^2\big]\,\mu^\alpha(\d y)\cr
&\le&\frac{1}{n}\int_{\R_>^N} \big( p_{2(t+s)}^\alpha(x,x)-1\big)\,\mu^\alpha_R(\d x)+ \Big(1-\frac{1}{n}\Big)\int_{\R_>^N}(P_s^\alpha b)^2\,\d\mu^\alpha.
\end{eqnarray}

Since $b$ has mean zero, by the Poincar\'{e} inequality implied by the curvature-dimension condition $\CD(1/2,\infty)$,  we have
$$\int_{\R_>^N}(P_s^\alpha b)^2\,\d\mu^\alpha\leq e^{-s}\int_{\R_>^N}  b^2\,\d\mu^\alpha,\quad s>0.$$
Note that, by \eqref{pf-main-3-2},
\begin{eqnarray*}
\int_{\R_>^N}  b^2\,\d\mu^\alpha&=&\frac{1}{\mu^\alpha(B_R)^2}\int_{\R_>^N}  \big(P_t^\alpha[\idx_{B_R}-\mu^\alpha(B_R)]\big)^2\,\d\mu^\alpha\\
&\leq&\frac{1}{\mu^\alpha(B_R)^2}\int_{\R_>^N}  \big(\idx_{B_R}-\mu^\alpha(B_R)\big)^2\,\d\mu^\alpha\\
&=&\frac{1-\mu^\alpha(B_R)}{\mu^\alpha(B_R)}.
\end{eqnarray*}
Hence
$$\int_{\R_>^N}(P_s^\alpha b)^2\,\d\mu^\alpha\leq e^{-s}\frac{1-\mu^\alpha(B_R)}{\mu^\alpha(B_R)},\quad s>0.$$
As above, letting $R=2\sqrt{C\log n}$, we have some constant $c>1$ such that
$$1-\mu^\alpha(B_R)\preceq\frac{1}{n^c},$$
for every large enough $n$. Hence, together with \eqref{pf-main-3-5}, there exists a constant $c>1$ such that
\begin{eqnarray}\label{pf-main-3-6}
&&\int_{\R_>^N}\E\big[P_s^\alpha (f_{n,R,t}-1)(y)\big]^2\,\mu^\alpha(\d y)\cr
&\preceq&\frac{1}{n}\int_{\R_>^N} \big( p_{2(t+s)}^\alpha(x,x)-1\big)\,\mu^\alpha_R(\d x) + \frac{1}{n^c}e^{-s},\quad t,s>0,
\end{eqnarray}
for every large enough $n$.

Thus, combining with \eqref{pf-main-3-0} and \eqref{pf-main-3-6}, we have some constants $C>0,\,c>1$ such that
\begin{equation}\begin{split}\label{pf-main-4}
\E\big[\W_2(\mu_{n,R,t},\mu^\alpha)^2\big]&\preceq\frac{1}{n}\int_0^\infty\int_{\R_>^N} \big( p_{2(t+s)}^\alpha(x,x)-1\big)\,\mu^\alpha_R(\d x)\,\d s + \frac{1}{n^c}\cr
&=\frac{1}{2n}\int_{2t}^\infty\int_{\R_>^N} \big( p_s^\alpha(x,x)-1\big)\,\mu^\alpha_R(\d x)\,\d s + \frac{1}{n^c},\quad t>0,
\end{split}\end{equation}
for $R=2\sqrt{C\log n}$  and every large enough $n$.

Now we shall estimate
$$\int_{2t}^\infty\int_{\R_>^N} \big( p_s^\alpha(x,x)-1\big)\,\mu^\alpha_R(\d x)\,\d s.$$
Let $x=(x_1, \cdots,x_N)\in \R_>^N$. By \eqref{HUP-1} and \eqref{HUP}, it is easy to see that, for each $j=1,\cdots,N$,
\begin{eqnarray}\label{DHUP}
p_t^{\alpha_j}(x_j,x_j)\preceq\begin{cases}
\frac{1}{(1-e^{-t})^{\alpha_j+1}}\exp\Big(-\frac{2e^{-t}x_j}{1-e^{-t}}\Big),\quad &0<x_j\le\frac{1-e^{-t}}{2e^{-t/2}},\\
\frac{(4e^{-t}x_j^2)^{-\alpha_j/2-1/4}}{e(1-e^{-t})^{1/2}}\exp\Big(-\frac{2e^{-t}x_j-2(e^{-t}x_j^2)^{1/2}}{1-e^{-t}}\Big),\quad& x_j>\frac{1-e^{-t}}{2e^{-t/2}}.
\end{cases}
\end{eqnarray}
Recalling that
$$p_t^\alpha(x,x)=\prod_{j=1}^Np_t^{\alpha_j}(x_j,x_j),$$
we have, for any large enough $R>0$,
\begin{equation}\begin{split}\label{VBR}
&\int_{B_R}p_t^\alpha(x,x)\d\mu^\alpha(x)\\
&\preceq\int_{B_R^1\times\cdots\times B_R^N}\Big(\prod_{j=1}^Np_t^{\alpha_j}(x_j,x_j)x_j^{\alpha_j} e^{-x_j}\Big)\,\d x_1\cdots\d x_N\\
&=\prod_{j=1}^N\int_{B_R^j}p_t^{\alpha_j}(x_j,x_j)x^{\alpha_j} e^{-x_j}\,\d x_j\\
&=\prod_{j=1}^N\Big(\int_0^{\frac{1-e^{-t}}{2e^{-t/2}}}p_t^{\alpha_j}(x_j,x_j)x^{\alpha_j} e^{-x_j}\,\d x_j+\int_{\frac{1-e^{-t}}{2e^{-t/2}}}^{\frac{R^2}{4}}p_t^{\alpha_j}(x_j,x_j)x_j^{\alpha_j} e^{-x_j}\d x_j\Big).
\end{split}\end{equation}

We need to estimate the two terms in the parentheses of \eqref{VBR}, and, for simplicity, we omit the subscript $j$ here. According to \eqref{DHUP}, on the one hand, it is easy to get that
\begin{equation}\begin{split}\label{VBR1}
&\int_0^{\frac{1-e^{-t}}{2e^{-t/2}}}p_t^\alpha(x,x)x^\alpha e^{-x}\,\d x\\
&\preceq \frac{1}{(1-e^{-t})^{\alpha+1}}\int_0^{\frac{1-e^{-t}}{2e^{-t/2}}}\exp\Big(-\frac{2e^{-t}x}{1-e^{-t}}\Big)x^\alpha e^{-x}\,\d x\\
&=\frac{1}{(1-e^{-t})^{\alpha+1}}\int_0^{\frac{1-e^{-t}}{2e^{-t/2}}}x^\alpha\exp\Big(-\frac{(1+e^{-t})x}{1-e^{-t}}\Big)\,\d x,
\end{split}\end{equation}
and on the other hand, when $\alpha\in[-\frac 1 2,\infty)$, we have
\begin{equation}\begin{split}\label{VBR2}
&\int_{\frac{1-e^{-t}}{2e^{-t/2}}}^{\frac{R^2}{4}}p_t^\alpha(x,x)x^\alpha e^{-x}\,\d x\\
&\preceq\int_{\frac{1-e^{-t}}{2e^{-t/2}}}^{\frac{R^2}{4}}\frac{(4e^{-t}x^2)^{-\alpha/2-1/4}}{e(1-e^{-t})^{1/2}}
\exp\Big(-\frac{2e^{-t}x-2(e^{-t}x^2)^{1/2}}{1-e^{-t}}\Big)x^\alpha e^{-x}\,\d x\\
&\preceq\frac{e^{t(\alpha/2+1/4)}}{(1-e^{-t})^{1/2}}\Big(\frac{1-e^{-t}}{e^{-t/2}}\Big)^{-\alpha-1/2}
\int_{\frac{1-e^{-t}}{2e^{-t/2}}}^{\frac{R^2}{4}}x^\alpha\exp\Big(-\frac{(1+e^{-t}-2e^{-t/2})x}{1-e^{-t}}\Big)\,\d x\\
&\preceq\frac{1}{(1-e^{-t})^{\alpha+1}}\int_{\frac{1-e^{-t}}{2e^{-t/2}}}^{\frac{R^2}{4}}x^\alpha
\exp\Big(-\frac{(1+e^{-t}-2e^{-t/2})x}{1-e^{-t}}\Big)\,\d x.
\end{split}\end{equation}
Combining \eqref{VBR1} and \eqref{VBR2} with \eqref{VBR}, for $\alpha\in [-1/2,\infty)^N$, we have the following estimate, i.e.,
\begin{eqnarray*}
\int_{B_R}p_t^\alpha(x,x)\, \mu(\d x)&\preceq& \prod_{j=1}^N\frac{1}{(1-e^{-t})^{\alpha_j+1}}\int_0^{\frac{R^2}{4}}x_j^{\alpha_j}
\exp\Big(-\frac{(1+e^{-t}-2e^{-t/2})x_j}{1-e^{-t}}\Big)\,\d x_j\\
&=&\prod_{j=1}^N\frac{1}{(1-e^{-t})^{\alpha_j+1}}\int_0^{\frac{1+e^{-t}-2e^{-t/2}}{1-e^{-t}}\frac{R^2}{4}}
\Big(\frac{1-e^{-t}}{1+e^{-t}-2e^{-t/2}}\Big)^{\alpha_j+1}x_j^{\alpha_j} e^{-x_j}\,\d x_j\\
&\preceq&\prod_{j=1}^N\frac{1}{(1-e^{-t})^{\alpha_j+1}\theta^{2(\alpha_j+1)}}\mu^{\alpha_j}(B_{\theta R}^j)\\
&=&\frac{1}{(1-e^{-t})^{\alpha_\star+N}\theta^{2(\alpha_\star+N)}}\mu^{\alpha}(B_{\theta R}),
\end{eqnarray*}
where $\theta:=\frac{1-e^{-t/2}}{\sqrt{1-e^{-t}}}$ and it obviously belongs to $(0,1)$ for every $t>0$. Thus,
\begin{eqnarray}\label{pf-main-3-7}\begin{split}
\int_{\R_>^N}[p_t^\alpha(x,x)-1]\,\mu^\alpha_R(\d x)&=\frac{1}{\mu^\alpha(B_R)}\int_{B_R}p_t^\alpha(x,x)\,\mu^\alpha(\d x)-1\cr
&\preceq\frac{1}{(1-e^{-t})^{\alpha_\star+N}\theta^{2(\alpha_\star+N)}}\cdot\frac{\mu^{\alpha}(B_{\theta R})}{\mu^{\alpha}(B_{R})}-1.
\end{split}\end{eqnarray}

If $\theta R\le 1$, then $t\le 2\log(\frac{R^2+1}{R^2-1})\leq \frac{8}{R^2}$ for large enough $R>0$. Then
\begin{eqnarray}\label{pf-main-3-8}\begin{split}
\frac{1}{(1-e^{-t})^{\alpha_\star+N}\theta^{2(\alpha_\star+N)}}\frac{\mu^\alpha(B_{\theta R})}{\mu^\alpha(B_R)}-1
&\preceq \frac{(\theta R)^{2(\alpha_\star+N)}}{(1-e^{-t})^{\alpha_\star+N}\theta^{2(\alpha_\star+N)}}\cr
&\preceq \frac{R^{2(\alpha_\star+N)}}{t^{\alpha_\star+N}},
\end{split}\end{eqnarray}
where we used the fact that $\mu^\alpha(B_R)\geq1/2$ for large enough $R>0$ and
\begin{eqnarray*}
\mu^\alpha(B_{\theta R})&=&\prod_{j=1}^N\int_0^{(\theta R)^2/4}\frac{x_j^{\alpha_j}e^{-x_j}}{\Gamma(1+\alpha_j)}\,\d x_j\\
&\preceq& \prod_{j=1}^N(\theta R)^{2(\alpha_j+1)}=(\theta R)^{2(\alpha_\star+N)}.
\end{eqnarray*}
If $\theta R>1$, then $e^{-t/2}< \frac{R^2-1}{R^2+1}\in(0,1)$, $t>0$.  Hence
\begin{equation*}\begin{split}
&\mu^\alpha(B_{\theta R})-(1-e^{-t/2})^{2(\alpha_\star+N)}\mu^\alpha(B_R)\\
&=\big[\mu^\alpha(B_{\theta R})-\mu^\alpha(B_R)\big]+\big[1-(1-e^{-t/2})^{2(\alpha_\star+N)}\big]\mu^\alpha(B_R)\\
&\le\mu^\alpha(B_R)\big[1-(1-e^{-t/2})^{2(\alpha_\star+N)}\big]\\
&\le 2(\alpha_\star+N)\mu^\alpha(B_R)e^{-t/2}.
\end{split}\end{equation*}
Then
\begin{equation}\label{pf-main-3-9}\begin{split}
&\quad\frac{1}{(1-e^{-t})^{\alpha_\star+N}\theta^{2(\alpha_\star+N)}}\frac{\mu^\alpha(B_{\theta R})}{\mu^\alpha(B_R)}-1\\
&=\frac{1}{(1-e^{-t/2})^{2(\alpha_\star+N)}}\frac{\mu^\alpha(B_{\theta R})}{\mu^\alpha(B_R)}-1\cr
&=\frac{\mu^\alpha(B_{\theta R})-(1-e^{-t/2})^{2(\alpha_\star+N)}\mu^\alpha(B_R)}{(1-e^{-t/2})^{2(\alpha_\star+N)}\mu^\alpha(B_R)}\cr
&\le\frac{2(\alpha_\star+N)e^{-t/2}}{(1-e^{-t/2})^{2(\alpha_\star+N)}}.
\end{split}\end{equation}

Combining \eqref{pf-main-3-7}, \eqref{pf-main-3-8} and \eqref{pf-main-3-9} with $T=2\log(\frac{R^2+1}{R^2-1})\leq\frac{8}{R^2}$ for large enough $R>0$ and $T\gg t$, we obtain that, when $\alpha_\star+N>1$,
\begin{equation*}\label{pf-main-5}\begin{split}
{\rm J}&:=\int_{2t}^\infty \int_{\R_>^N}[p_s(x,x)-1]\,\d\mu^\alpha_R(x)\d s \\
&=\int_{2t}^T\int_{\R_>^N}[p_s(x,x)-1]\,\d\mu^\alpha_R(x)\d s+\int_T^\infty\int_{\R_>}[p_s(x,x)-1]\,\d\mu^\alpha_R(x)\d s\\
&\preceq \int_{2t}^T \frac{R^{2(\alpha_\star+N)}}{s^{\alpha_\star+N}}\,\d s+\int_T^\infty \frac{e^{-s/2}}{(1-e^{-s/2})^{2(\alpha_\star+N)}}\,\d s\\
&\preceq \frac{R^{2(\alpha_\star+N)}}{t^{\alpha_\star+N-1}}+R^{4(\alpha_\star+N)-2};
\end{split}\end{equation*}
when $\alpha_\star+N=1$,
$${\rm J}\preceq R^2\log\left(\frac 1 t\right)+R^2;$$
when $\alpha_\star+N\in(1/2,1)$,
$${\rm J}\preceq R^{2(\alpha_\star+N)}T^{1-\alpha_\star-N}+R^{4(\alpha_\star+N)-2}\preceq R^{4(\alpha_\star+N)-2};$$
when $\alpha_\star+N=1/2$,
$${\rm J}\preceq R\sqrt{T}+\log\left(\frac{1}{T}\right)\preceq 1+\log R^2.$$
Thus, together with \eqref{pf-main-4}, we deduce that
\begin{equation}\begin{split}\label{part3}
&\E\big[\W_2(\mu_{n,R,t},\mu^\alpha)^2\big]\preceq\frac{1}{n}{\rm J} + \frac{1}{n^c}\cr
&\preceq\begin{cases}\medskip
\frac{1}{n^{c}} +\frac 1 {n}\big[\frac{R^{2(\alpha_\star+N)}}{t^{\alpha_\star+N-1}}+R^{4(\alpha_\star+N)-2}\big],\quad&{\alpha_\star\in(1-N,\infty),}\\\medskip
\frac{1}{n^{c}} +\frac 1 {n}\big[R^2\log\left(\frac 1 t\right)+R^2\big],\quad&{\alpha_\star=1-N,}\\\medskip
\frac{1}{n^{c}} +\frac 1 {n}R^{4(\alpha_\star+N)-2},\quad&{\alpha_\star\in (1/2-N,1-N),}\\\medskip
\frac{1}{n^{c}} +\frac 1 {n}\big(1+\log R^2\big),\quad&{\alpha_\star=1/2-N,}
\end{cases}
\end{split}\end{equation}
which finishes the proof of \textsc{Part} (3).

Finally, combining \eqref{pf-main-1} with the estimates in \eqref{part1}, \eqref{part2} and \eqref{part3} together, we arrive at
\begin{equation}\label{pf-main-6}
\E\big[\W_2(\mu_n,\mu^\alpha)^2\big]\preceq\begin{cases}\medskip
\frac{1}{n^{c}}+t+\frac 1 {n}\big[\frac{R^{2(\alpha_\star+N)}}{t^{\alpha_\star+N-1}}+R^{4(\alpha_\star+N)-2}\big],\quad&{\alpha_\star\in(1-N,\infty),}\\\medskip
\frac{1}{n^{c}}+t+\frac 1 {n}\big[R^2\log\left(\frac 1 t\right)+R^2\big],\quad&{\alpha_\star=1-N,}\\\medskip
\frac{1}{n^{c}}+t+\frac 1 {n}R^{4(\alpha_\star+N)-2},\quad&{\alpha_\star\in (1/2-N,1-N),}\\\medskip
\frac{1}{n^{c}}+t+\frac 1 {n}\big(1+\log R^2\big),\quad&{\alpha_\star=1/2-N,}
\end{cases}
\end{equation}
for any $R\asymp \sqrt{\log n}$ and large $n$. Optimizing \eqref{pf-main-6} in $t>0$, we complete the proof of \eqref{main-thm-1} for every large enough $n$.
\end{proof}

\subsection*{Acknowledgment}\hskip\parindent
The authors would like to thank the referee for corrections and helpful comments, and acknowledge financial supports from the National Key R\&D Program of China (Grant No. 2022YFA1006000)
and the National Natural Science Foundation of China (No.11831014).

\end{document}